\documentclass[11pt]{amsart}

\usepackage{amsfonts}
\usepackage{amssymb}
\usepackage{graphicx}
\usepackage{amsmath}
\usepackage{latexsym}
\usepackage{amscd}
\usepackage{xypic}
\usepackage{mathrsfs}%scrfont}
\usepackage{enumitem} %http://ctan.org/pkg/enumitem
\usepackage{braket}
\usepackage[margin=1.4in]{geometry}
\usepackage{comment}
\usepackage{bm}
\usepackage{marginnote}
\usepackage{color}
\usepackage[normalem]{ulem}

\usepackage{imakeidx}
\makeindex[title=Index of Notation]

\usepackage{hyperref}

\allowdisplaybreaks 

\numberwithin{equation}{section}
\newtheorem{theorem}{Theorem}[section]

\newtheorem{corollary}[theorem]{Corollary}

\newtheorem{proposition}[theorem]{Proposition}

\newtheorem*{theorem*}{Theorem}

\theoremstyle{definition}
\newtheorem{definition}[theorem]{Definition}

\theoremstyle{remark}
\newtheorem{remark}{Remark}

\theoremstyle{remark}
\newtheorem{claim}[theorem]{Claim}

\theoremstyle{definition}

\newcommand{\NN}{\mathbb{N}}

\newcommand{\RR}{\mathbb{R}}
\renewcommand{\SS}{\mathbb{S}}

\renewcommand{\cH}{\mathcal H}

\newcommand{\cM}{\mathcal M}

\newcommand{\cS}{\mathcal S}

\newcommand{\bH}{\mathbf{H}}

\newcommand{\bx}{\mathbf{x}}
\newcommand{\by}{\mathbf{y}}

\newcommand{\bOh}{\mathbf{0}}

\DeclareMathOperator{\supp}{supp}

\newcommand{\eps}{\varepsilon}

\DeclareMathOperator{\sing}{sing}
\DeclareMathOperator{\reg}{reg}

\title{Revisiting generic mean curvature flow in $\RR^3$}

\author{Otis Chodosh} 
\address{OC: Department of Mathematics, Bldg.\ 380, Stanford University, Stanford, CA 94305, USA}
\email{ochodosh@stanford.edu}

\author{Kyeongsu Choi}
\address{KC: School of Mathematics, Korea Institute for Advanced Study, 85 Hoegiro, Dongdaemun-gu, Seoul 02455, Republic of Korea}
\email{choiks@kias.re.kr}
\author{Christos Mantoulidis} 
\address{CM: Department of Mathematics, Rice University, Houston, TX 77005, USA}
\email{christos.mantoulidis@rice.edu}
\author{Felix Schulze}
\address{FS: Department of Mathematics, Zeeman Building, University of Warwick, Gibbet Hill Road, Coventry CV4 7AL,
UK}
\email{felix.schulze@warwick.ac.uk} 

\begin{document}

\begin{abstract}
Bamler--Kleiner recently proved a multiplicity-one theorem for mean curvature flow in $\RR^3$ and combined it with the authors' work on generic mean curvature flows to fully resolve Huisken's genericity conjecture. In this paper we show that a short density-drop theorem plus the Bamler--Kleiner multiplicity-one theorem for tangent flows at the first nongeneric singular time suffice to resolve Huisken's conjecture -- without relying on the strict genus drop theorem for one-sided ancient flows previously established by the authors. \end{abstract}

\maketitle

\section{Introduction}

Mean curvature flow is the gradient flow of area. A family of hypersurfaces $M(t) \subset \RR^{n+1}$ is flowing by mean curvature flow, provided
\begin{equation} \label{eq:mcf}
\left(\tfrac{\partial}{\partial t} \bx \right)^{\perp} = \bH_{M(t)}(\bx).
\end{equation}
Here, $\bH_{M(t)}(\bx)$ denotes the mean curvature vector of $M(t)$ at $\bx$. By a simple application of the maximum principle, the flow $M(t)$ becomes singular in finite time whenever $M(0) \subset \RR^{n+1}$ is closed and embedded. There are several weak notions of mean curvature flow that allow to flow through singularities. In this paper, we use the notion of cyclic, unit-regular, integral Brakke flows. See Section \ref{sec:background} for definitions and references. %See Remark \ref{rema:implications} for implications for other weak notions of flows.

A well-known conjecture of Huisken (see \cite[\#8]{Ilmanen:problems}) posited that generic mean curvarture flows in $\RR^3$ encounter only spherical and cylindrical singularities. Groundbreaking progress toward Huisken’s conjecture was made in the work of Colding–Minicozzi \cite{ColdingMinicozzi:generic}, who proved that spheres and cylinders are, in a sense, the only linearly stable singularity models for mean curvature flow (in all dimensions). They also showed how to construct ``broken'' mean curvature flows that avoid closed, multiplicity-one, non-spherical singular points. Subsequently, the present authors developed a one-sided perturbation and genus-drop mechanism  \cite{CCMS:generic1, CCS:generic2} which proves Huisken's conjecture under the hypothesis that all singularities arise with multiplicity one. This assumption that singularities always have multiplicity one (Ilmanen's multiplicity one conjecture \cite[\#2]{Ilmanen:problems}) was recently proven to always hold in a remarkable work of Bamler--Kleiner \cite{BamlerKleiner}. As such, when combined with our previous work, this resolved Huisken's conjecture. 

In this paper, we develop a shorter replacement for our previous works that, when combined with the work of Bamler--Kleiner, gives a new proof of Huisken's conjecture:

\begin{theorem}\label{theo:main}
Let $M^\circ \subset \RR^3$ be a closed embedded surface. There exist arbitrarily small $C^\infty$ graphs $M$ over $M^\circ$ so that any cyclic unit-regular integral Brakke flow starting at $M(0) := M$ only has only multiplicity-one spherical and cylindrical singularities.
\end{theorem}

The main contribution of this paper is to circumvent the genus-drop theorem; we employ, in its place, a rather short density-drop theorem inspired by our other work on the generic regularity of higher-dimensional mean curvature flow under entropy bounds \cite{CCMS:generic-low-ent, CMS:generic-low-ent-2} and on the generic regularity of area-minimizing currents \cite{CMS:minimizing-9-10, CMS:minimizing-codim}. Its adaptation to our current setting is nonetheless quite subtle: on the one hand, the argument is global in spacetime in nature, while on the other hand, it is sensitive to us only invoking the multiplicity-one result at the \emph{first} nongeneric time.

\begin{remark} \label{rema:implications}
	 Let $M \subset \RR^3$ be closed embedded surface and $\cM$ be a cyclic unit-regular integral Brakke flow starting from $M$. Such $\cM$ can be produced by Ilmanen's elliptic regularization scheme \cite{Ilmanen:elliptic}. Assume that $\cM$ has the property arranged in Theorem \ref{theo:main}, i.e.:
	\begin{equation}  \label{eq:implications} \tag{$\dagger$}
		\cM \text{ only has multiplicity-one spherical and cylindrical singularities}.
	\end{equation}
	Here are some well-known implications of $(\dagger)$:
	\begin{enumerate}
		\item[(a)] The work of Hershkovits--White \cite{HershkovitsWhite:nonfattening}, combined with the resolution of the mean convex neighborhood conjecture by Choi--Haslhofer--Hershkovits \cite{ChoiHaslhoferHershkovits}, yields that the level-set flow of $M$ is non-fattening and that the innermost and outermost flows starting at $M$ agree, with their spacetime support equal to the level-set flow of $M$. Moreover, $\cM$ is the \emph{unique} cyclic unit-regular integral Brakke flow starting at $M$ by \cite[Appendix G]{CCMS:generic1}. Therefore $\cM$, the innermost, and the outermost flows starting at $M$ all coincide and are almost-regular in the sense of Bamler--Kleiner \cite[Definition 2.6]{BamlerKleiner} (cf.\ \cite[Lemma 7.8]{BamlerKleiner}).
		\item[(b)] The work of Daniels-Holgate \cite{Daniels-Holgate} implies that there exists a mean curvature flow with surgery starting from $M$ (this result uses the results in point (a) as well as the surgery results in the mean convex case \cite{Brendle:inscribed-sharp,BrendleHuisken:R3,HaslhoferKleiner:estimates,HaslhoferKleiner:surgery}).
		\item[(c)] Consider any sequence $M_i \to M$, say in $C^1$, and a corresponding sequence $\cM_i$ of cyclic unit-regular integral Brakke flows starting at $M_i$. Since the class of cyclic unit-regular integral Brakke flows is closed under weak convergence of Brakke flows, it follows from (a) above that $\cM_i \rightharpoonup \cM$. Furthermore, it follows from the classification of low entropy shrinkers \cite{BernsteinWang:TopologicalProperty} (or, alternatively, the proof of the mean convex neighborhood conjecture--see \cite[Propsition 2.3]{SchulzeSesum}) that, for sufficiently large $i$, $\cM_i$ also only has only multiplicity-one spherical and cylindrical singularities. In particular, the set of $M$ with corresponding $\cM$ satisfying \eqref{eq:implications} is both dense (per Theorem \ref{theo:main}) and open.
	\end{enumerate}
\end{remark}

\subsection{Other works}
A number of works other than \cite{BamlerKleiner} have been very influential, if not outright invoked, in Theorem \ref{theo:main}. We highlight in particular some works of Bernstein, Brendle, Colding, Haslhofer, Hershkovits, Ilmanen, Minicozzi, Wang, and White in approximate chronological order: \cite{Ilmanen:elliptic, White:topology-weak, White:size, White:nature, White:cyclic, ColdingMinicozzi:generic, ColdingMinicozziIlmanenWhite, ColdingIlmanenMinicozzi, ColdingMinicozzi:sing-generic, Brendle:genus0, Wang:ends-conical, BernsteinWang:TopologicalProperty, ChoiHaslhoferHershkovits, HershkovitsWhite:nonfattening}.

\subsection{The genus-drop approach} \label{subsec:genus-drop}
We recall the genus-drop approach of  \cite{CCMS:generic1, CCS:generic2}. This is not used anywhere else in this paper. 

Consider a one-parameter family of mean curvature flows $\cM_s$ obtained by flowing monotone (``one-sided'') perturbations $M_s$, $s = o(1)$, of the initial surface $M_0 := M$. Consider a family of Brakke flows $(\cM_s)$ starting at $(M_s)$. Assume that $\cM_0$ develops a multiplicity-one singularity at $(\bx_0, t_0)$ which is \emph{not} spherical or cylindrical. Any tangent flow of $\cM_0$ at $(\bx_0, t_0)$ is a multiplicity-one homothetically shrinking flow that is not a sphere or a cylinder (see Section \ref{subsec:shrinkers}). By the avoidance principle, the corresponding blow up limits of nearby $\cM_s$'s at $(\bx_0, t_0)$ will be ancient flows on either side of the central shrinking flow. 

The first step is to now study the rescaled mean curvature flows of these blow ups. By developing a PDE classification tool for solutions with fixed-sign speed (positive or negative) of the rescaled mean curvature flow on one side of non-spherical and non-cylindrical shrinkers, we show that the blow-up limits of the nearby $\cM_s$ move monotonically in the sense of rescaled flows (i.e.~in a shrinker mean convex sense in the original scale). This built on previous work of the authors in \cite{ChoiMantoulidis}, inspired by \cite{ADS, ADS2}. One then  deduces from this, as in \cite{White:size}, that the $(\cM_s)_{s \neq 0}$ only develop multiplicity-one spherical and cylindrical singularities and have genus zero \emph{near} $(\bx_0, t_0)$. 

The second step is to prove a globalized and iterable genus-drop result. By Brendle's genus zero classification theorem \cite{Brendle:genus0}, and a localization of White's genus monotonicity \cite{White:topology-weak}, one may show that there has been a (strict) genus drop of $\cM_s(t)_{s \neq 0}$ on approach to time $t=t_0$. This can be iterated finitely many times, bounded by the genus of $M_0 := M$, to ultimately get:
\begin{quote}
\emph{For generic closed embedded surfaces $M(0) \subset \RR^3$, mean curvature flow has only spherical and cylindrical singularities for as long as its singularities have multiplicity one.}	
\end{quote}
In particular, Huisken's conjecture was reduced to Ilmanen's multiplicity-one conjecture (see \cite[\#2]{Ilmanen:problems}), which was eventually resolved by Bamler--Kleiner \cite{BamlerKleiner}. 

\subsection{The density-drop approach} \label{subsec:density-drop}

As before, we let $\cM_s$ be a monotone one-parameter family of flows, and we study blow-ups of this family at multiplicity-one singular points modeled by non-spherical, non-cylindrical shrinkers. 

Let $\Sigma$ be any such shrinker. The key is Proposition \ref{prop:geometric-characterization}, a qualitative geometric result characterizing such $\Sigma$: any translate of its shrinking spacetime track must ``cross'' its original, untranslated shrinking track. This is a special feature of shrinkers $\Sigma$ that are not spheres or cylinders. Clearly, it is not true for spheres and cylinders.

Our density-drop result, Proposition \ref{prop:density-drop}, quantifies Proposition \ref{prop:geometric-characterization}. Let $\Theta_\Sigma$ be the Huisken density of the singular point modeled on $\Sigma$. Proposition \ref{prop:density-drop} shows that every ancient flow to one side of the shrinking flow of $\Sigma$, with entropy no larger than $\Theta_\Sigma$, has Huisken density $\leq \Theta_\Sigma - \eta_0$ at all singular points (except the origin). Here, $\eta_0$ depends only on background a priori genus and entropy bounds, but not $\Sigma$. It follows from Proposition \ref{prop:density-drop}, and a monotonicity argument, that all sufficiently nearby points in the pre-blow-up picture have Huisken density $\leq \Theta_\Sigma - \eta_0$. Thus, our central singular point is isolated among singular points with density $\geq \Theta_\Sigma - \tfrac12 \eta_0$. 

We now globalize. For each $\cM_s$, $s=o(1)$, we may assume the existence of singular points that are not modeled by multiplicity-one spheres or cylinders. (Otherwise, we are done.) Let us then restrict our attention to the first (ordered by time) such singular point(s). Existence follows by work of Bernstein--Wang \cite{BernsteinWang:TopologicalProperty} (see Proposition \ref{prop:stability-generic}). The work of Ilmanen \cite{Ilmanen:singularities} and Bamler--Kleiner's resolution of the multiplicity-one conjecture \cite{BamlerKleiner} (see Theorem \ref{theo:ilmanen-bamler-kleiner}) imply that these singular points are modeled by multiplicity-one shrinkers that are not spheres or cylinders. Therefore, Proposition \ref{prop:density-drop} (density-drop) applies and guarantees that these points are isolated as long as we suitably stratify by density. In particular, by elementary considerations, the entire set of such points turns out to be finite, so we can clearly choose $\cM_s$'s that miss it.

\subsection{Comparison} The genus-drop approach to resolve Huisken's conjecture is significantly longer than the density-drop in the current paper, but it also gives a refined local picture. In particular, together with the Bamler--Kleiner multiplicity-one theorem and their compactness theorem for bounded almost regular flows (\cite[Theorem 1.7]{BamlerKleiner}), inner/outermost flows experience strict genus drop after every nongeneric singular point (thus bounding the number of nongeneric singular points by the initial genus) see, e.g., the proof of \cite[Theorem 1.9 (c)]{BamlerKleiner}.

\subsection{Organization} In Section \ref{sec:background}, we recall some standard background notions about smooth and weak mean curvature flows that are part of our setup. In Section \ref{sec:ingredients}, we list the various technical ingredients from \cite{ColdingMinicozzi:compactness-shrinkers, ChengZhou:volume-estimate, BernsteinWang:TopologicalProperty, CCMS:generic1, CCMS:generic-low-ent, BamlerKleiner} that are necessary for our proof. In Section \ref{sec:proof} we give the proof of Theorem \ref{theo:main}.

\subsection{Acknowledgements} OC was supported by a Terman Fellowship and an NSF grant (DMS-2304432). KC was supported by the KIAS Individual Grant MG078902. CM was supported by grant NSF DMS 2403728. We thank Richard Bamler and Bruce Kleiner for their interest in this work, and the anonymous referees for helpful suggestions.

\section{Background} \label{sec:background}

\subsection{Spacetime} \label{subsec:spacetime}

It will prove to be very convenient to study our flows in spacetime, $\RR^3 \times \RR$. To that end, we will refer to points such as
\[ X = (\bx, t) \in \RR^3 \times \RR \]
as spacetime points, and we will utilize the time function
\[ \mathfrak{t}(\bx, t) := t. \]
 Furthermore we denote parabolic distance in spacetime using
\[d_p(X,Y) = d_p\big((\bx, t), (\by,s)\big):= \sqrt{ \|\bx-\by\|^2+ |t-s|}, \]
and the natural parabolic dilation around the spacetime origin $(\bOh,0)$, by $\lambda > 0$, using
\[
\operatorname{ParDil}_\lambda : \RR^3 \times \RR \to \RR^3 \times \RR, \; \operatorname{ParDil}_\lambda(\bx, t) = (\lambda^{-1} \bx, \lambda^{-2} t).
\]

\subsection{$F$-functional and density}

Let us begin by recalling a few basic notions in the smooth setting. First, for smooth surfaces $M \subset \RR^{3}$ one has the $F$-functional
\begin{equation} \label{eq:f-functional}
	F(M) = (4\pi)^{-1} \int_M e^{-\tfrac14 |\bx|^2} \, d\cH^2(\bx),
\end{equation}
which in turn gives rise to the Colding--Minicozzi entropy 
via
\begin{equation} \label{eq:f-functional-entropy}
	\lambda(M) = \sup_{\substack{\bx_0\in\RR^{3}\\t_0>0}} F \left( \frac{1}{\sqrt{t_0}} (M - \bx_0) \right).
\end{equation} 
Moreover, for a smooth mean curvature flow with bounded area ratios $\cM : t \mapsto M(t)$, the $F$-functional also gives rise to the density function, defined as
\begin{equation} \label{eq:density-r}
	\Theta_\cM(X, r) = F\left( \frac{1}{r} (M(t - r^2) - \bx) \right), \; X \in \RR^{3} \times \RR, \; r > 0,
\end{equation}
which is nondecreasing in $r$ by Huisken's well-known monotonicity formula \cite{Huisken:asymptotic}. In particular,
\begin{equation} \label{eq:density}
	\Theta_\cM(X) = \lim_{r \to 0} \Theta_\cM(X, r), \; X \in \RR^{3} \times \RR,
\end{equation}
is well-defined.

\subsection{Shrinkers} \label{subsec:shrinkers}

The spheres and cylinders referred to in Theorem \ref{theo:main} are types of self-similarly shrinking (or shrinkers, for short) singularity models for mean curvature flow. 

\begin{definition}
A surface $\Sigma \subset \RR^3$ is said to be a shrinker if it satisfies $\bH + \tfrac12 \bx^\perp = \bOh$, where $\bH$ is the mean curvature vector of $\Sigma$.
\end{definition}

Equivalently, $\Sigma$ is a shrinker if $t\mapsto \sqrt{-t}\, \Sigma$ is a mean curvature flow for $t<0$. 

\begin{definition}
We denote
\begin{align*} 
	\cS & := \{ \Sigma \subset \RR^3 \text{ embedded shrinker with } \lambda(\Sigma) < \infty \}, \\
	\cS^* & := \cS \setminus \{ \text{planes} \}, \\
	\cS^\textnormal{gen} & := \{ \SS^2(2) \} \cup \left\{ O(\SS^1(\sqrt{2}) \times \RR^{1}) \in \cS : O \in O(3) \right\}.
\end{align*}
\end{definition}

The elements spheres and cylinders referred to in Theorem \ref{theo:main} are the elements of $\cS^{\textnormal{gen}}$, whose radii were chosen so that $\cS^\textnormal{gen} \subset \cS$.

\subsection{Brakke flows} \label{subsec:brakke.flows}

For our study of singular mean curvature flows, we need to work with weaker objects than smooth hypersurfaces, namely varifolds in $\RR^3$ (see \cite{Simon:GMT}). The corresponding weak notion of mean curvature flows are Brakke flows in $\RR^3$, which we recall below (cf.\ \cite{Brakke,Ilmanen:elliptic}). All our varifolds and Brakke flows are always taken to be $2$-dimensional.

\begin{definition} 
	A ($2$-dimensional) \emph{integral Brakke flow} in $\RR^{3}$ is a $1$-parameter family of Radon measures $(\mu(t))_{t \in I}$ over an interval $I \subset \RR$ so that:
	\begin{enumerate}
		\item For almost every $t \in I$, there exists an integral $n$-dimensional varifold $V(t)$ with $\mu(t) = \mu_{V(t)}$ so that $V(t)$ has locally bounded first variation and has mean curvature $\bH$ orthogonal to $\textrm{Tan}(V(t),\cdot)$ almost everywhere.
		\item For a bounded interval $[t_1,t_2] \subset I$ and any compact set $K\subset \RR^{3}$, 
		\[
			\int_{t_1}^{t_2}\int_K (1+|\bH|^2) d\mu(t) dt < \infty.
		\]
		\item If $[t_1,t_2] \subset I$, $f \in C^{1}_{c}(\RR^{3}\times [t_1,t_2])$, and $f\geq 0$, then Brakke's inequality holds:
		\[
			\int f(\cdot,t_{2}) \, d\mu(t_{2}) - \int f(\cdot,t_{1}) \, d\mu(t_{1}) \leq \int_{t_{1}}^{t_{2}} \int\left( - |\bH|^{2} f + \bH \cdot \nabla f + \tfrac{\partial }{\partial t} f \right) \, d\mu(t) \, dt.
		\]
	\end{enumerate}
\end{definition}

We will often write $\cM$ for a Brakke flow $(\mu(t))_{t \in I}$, with the understanding that we're referring to the family $I \ni t\mapsto \mu(t)$ of measures satisfying Brakke's inequality. In this case we write $\cM(t)$ to mean $\mu(t)$.

In this weaker setting setting, \eqref{eq:f-functional}, \eqref{eq:f-functional-entropy} extend to varifolds by looking at their induced Radon measures, and the monotonicity underlying \eqref{eq:density-r}, \eqref{eq:density} extends to Brakke flows in $\RR^{3}$ with bounded area ratios (cf.\ \cite[Lemma 7]{Ilmanen:singularities}). 

Furthermore, we slightly abuse notation and still denote with $\operatorname{ParDil}_\lambda$ the standard parabolic dilation of a Brakke flow around the spacetime origin $(\bOh,0)$ by a factor $\lambda>0$. This is obtained by pushing forward using the parabolic dilation map from Section \ref{subsec:spacetime}. 

\subsection{$F$-stationary varifolds}

The shrinker condition is also equivalent to the following (see \cite{Ilmanen:Trieste}, \cite[\S3]{ColdingMinicozzi:generic}):
\begin{itemize}
\item $\Sigma$ is a minimal hypersurface for the metric $e^{-\frac{1}{4} |\bx|^{2}}g_{\RR^{3}}$, or
\item $\Sigma$ is a critical point of the $F$-functional from \eqref{eq:f-functional} among compactly supported deformations, as well as translations and dilations.
\end{itemize}

Definition \ref{defi:f-stationary-varifold} below uses precisely this insight to correctly generalize the notion of shrinkers to the varifold setting.

\begin{definition} \label{defi:f-stationary-varifold}
	A varifold $V$ in $\RR^{3}$ is called $F$-stationary if is stationary with respect to the conformally flat metric $e^{-|\bx|^2/4} g_{\RR^{3}}$ on $\RR^{3}$.
\end{definition}

Of course, the immediate relevance of this in taking tangent flows:

\begin{definition} \label{defi:tangent-flow}
	Let $\cM$ be an integral Brakke flow in $\RR^3$, $X \in \RR^3 \times \RR$, and $\lambda_i \to 0$. If 
	\[ \tilde \cM_i := \operatorname{ParDil}_{\lambda_i}(\cM - X) \rightharpoonup \tilde \cM, \]
	then we call $\tilde \cM$ a tangent flow to $\cM$ at $X$. It follows from the monotonicity formula that, for $t < 0$, $\tilde \cM$ coincides with a shrinking integral Brakke flow
	\[ \cM_V(t) = \begin{cases} \sqrt{-t} V & t < 0, \\ 0 & t \geq 0, \end{cases} \]
	for some $F$-stationary integral varifold $V$ in $\RR^3$. See \cite[Lemma 8]{Ilmanen:singularities} or \cite{White:stratification}.
\end{definition}

\subsection{Regular, singular, generic singular points} \label{subsec:reg-sing-gensing}

Briefly recall that:

\begin{definition} 
	An integral Brakke flow $\cM$ is said to be:
	\begin{itemize}
		\item cyclic if, for a.e.\ t, $\mu(t) = \mu_{V(t)}$ for an integral varifold $V(t)$ whose associated rectifiable mod-2 flat chain $[V(t)]$ has $\partial [V(t)] = 0$ (see \cite{White:cyclic});
		\item unit-regular if $\Theta_\cM(X) = 1$ implies that there exists a forward-backward parabolic ball around $X$ in which $\cM$ is a smooth connected multiplicity-one flow.
	\end{itemize}
\end{definition}

\begin{definition}[{\cite[Definition 1.6]{CCMS:generic-low-ent}}] \label{defi:reg-sing-gen}
	Let $\cM$ be a unit-regular integral Brakke flow in $\RR^3$.
		\begin{enumerate}
			\item[(a)] We denote by $\reg \cM$ the set of $X \in \supp \cM$ for which there exists a forward-backward parabolic ball centered at $X$ inside of which $\cM$ is a smooth connected multiplicity-one flow.
			\item[(b)] We denote $\sing \cM = \supp \cM \setminus ((\supp M(0) \times \{0\}) \cup \reg \cM)$.
			\item[(c)] We denote by $\sing_{\textnormal{gen}} \cM$ the set of $X \in \sing \cM$ so that all\footnote{If some tangent flow is a multiplicity one element of $\cS^{\textnormal{gen}}$, then all are (\cite{ColdingIlmanenMinicozzi,ColdingMinicozzi:uniqueness-tangent-flow}; cf.\ \cite{BernsteinWang:high-mult-unique}).} tangent flows to $\cM$ at $X$ are, for $t < 0$, coincident with some $\cM_\Sigma$, $\Sigma \in \cS^\textnormal{gen}$ (see Definition \ref{defi:tangent-flow}).
		\end{enumerate}
\end{definition}

\section{Technical ingredients} \label{sec:ingredients}

\subsection{Shrinker compactness} 

The following compactness theorem follows from the work of Colding--Minicozzi \cite{ColdingMinicozzi:compactness-shrinkers} and Cheng--Zhou \cite{ChengZhou:volume-estimate}:

\begin{theorem} \label{theo:shrinker-compactness}
	Fix $g \in \NN$, $\Lambda > 0$. If we have a sequence $\Sigma_i \in \cS$ with
	\[ \operatorname{genus}(\Sigma_i) \leq g, \; F(\Sigma_i) \leq \Lambda, \]
	then, after passing to a subsequence, we can find $\Sigma \in \cS$ such that
	\[ \Sigma_i \to \Sigma \text{ in } C^\infty_{\textnormal{loc}}, \; \operatorname{genus}(\Sigma) \leq g, \; F(\Sigma_i) \to F(\Sigma). \]
\end{theorem}

\subsection{Entropy gap} \label{subsec:entropy-gap}

The entropy gap result below, due to Bernstein--Wang \cite{BernsteinWang:TopologicalProperty}, establishes a definite entropy gap between self-shrinking cylinders and all other shrinkers in $\RR^3$. We also include below Stone's explicitly computation in \cite{Stone} of entropies of spheres and cylinders:

\begin{theorem}  \label{theo:entropy-gap}
	There exists a universal $\delta_0 > 0$ such that every $\Sigma \in \cS^* \setminus \cS^{\textnormal{gen}}$ satisfies:\footnote{Note that $\lambda$ is invariant under rescaling, so we need not rescale our spherical factors.}
\[
	1 = \lambda(\RR^2) < \lambda(\SS^2) = \tfrac{4}{e} \approx 1.47 < \lambda(\SS^1 \times \RR) = \sqrt{\tfrac{2\pi}{e}} \approx 1.52 \leq \lambda(\Sigma) - \delta_0.
\]
\end{theorem}

The gap result builds on previous important work, including Brendle's classification of genus zero self-shrinkers \cite{Brendle:genus0} and Colding--Ilmanen--Minicozzi--White's proof that the round sphere has minimal entropy among all closed self-shrinkers \cite{ColdingMinicozziIlmanenWhite}. 

\subsection{Stability of generic singularities}\label{subsec:stability-generic}

We recall the following result that, in $\RR^3$, follows from elementary density upper semicontinuity considerations and the Bernstein--Wang entropy gap from Theorem \ref{theo:entropy-gap}. 

\begin{proposition}[{\cite[Proposition A.1]{CCMS:generic-low-ent}}] \label{prop:stability-generic}
Consider cyclic unit-regular integral Brakke flows $\cM_i\rightharpoonup \cM$ in $\RR^3$ with $X_i \in \sing\cM_i$ converging to $X \in \sing_{\textnormal{gen}} \cM$. Then, $X_i \in \sing_\textnormal{gen}\cM_i$ for all sufficiently large $i$.
\end{proposition}

See also \cite[Propsition 2.3]{SchulzeSesum}.

\subsection{Frankel property for $F$-stationary varifolds} First recall the following Frankel theorem for $F$-stationary varifolds from \cite[Corollary D.4]{CCMS:generic1}:

\begin{theorem} \label{theo:frankel-shrinker}
	If $V, V'$ are $F$-stationary varifolds, then $\supp V \cap \supp V' \neq \emptyset$.
\end{theorem}

The following is essentially contained in the proof of \cite[Theorem 7.17 (5)]{CCMS:generic1}.

\begin{proposition} \label{prop:frankel-shrinker}
	Fix $\Sigma \in \cS$ and let $\Omega \subset \RR^3$ be either component of $\RR^3 \setminus \Sigma$. If $V$ is an $F$-stationary integral varifold in $\RR^3$ with
	\begin{equation} \label{eq:frankel-shrinker}
		\supp V \subset \bar \Omega,
	\end{equation}
	then $V = k\Sigma$ for some $k \in \NN$.\footnote{Here and throughout, we implicitly identify $\Sigma$ with its induced integral varifold.}
\end{proposition}
\begin{proof}
	It follows from \eqref{eq:frankel-shrinker} and the Solomon--White maximum principle for stationary varifolds \cite{SolomonWhite} that all components of $\supp V$ are either disjoint from $\partial \Omega = \Sigma$ or they coincide with it. Theorem \ref{theo:frankel-shrinker} rules out the former case, so $\supp V = \Sigma$. The result that $V$ corresponds to an integral multiple of $\Sigma$ now follows from the constancy theorem for stationary varifolds \cite{Simon:GMT} (applied in the conformal metric) and $\Sigma$'s smoothness.
\end{proof}

\subsection{Geometric characterization of generic singularities}

The following characterization of generic shrinkers was central to \cite{CCMS:generic-low-ent} (see also Colding--Minicozzi's classification of $\cS^{\textrm{gen}}$ as the unique linearly stable self-shrinkers \cite{ColdingMinicozzi:generic}).

\begin{proposition}[{\cite[Proposition 2.2]{CCMS:generic-low-ent}}] \label{prop:geometric-characterization}
	Fix $\Sigma \in \cS^*$ and let $\Omega \subset \RR^3$ be either component of $\RR^3 \setminus \Sigma$. Assume that there exists $X_0 \in (\RR^3 \times \RR) \setminus (\bOh, 0)$ such that
	\begin{equation} \label{eq:geometric-characterization}
		(X_0 + \supp \cM_\Sigma) \cap \mathfrak{t}^{-1}((-\infty, \min\{0, \mathfrak{t}(X_0)\})) \subset \cup_{t < 0} \sqrt{-t} \bar \Omega.
	\end{equation}
	Then, $\Sigma \in \cS^{\textnormal{gen}}$. 
\end{proposition}

In short, Proposition \ref{prop:geometric-characterization} characterizes $\Sigma \in \cS^{\textrm{gen}}$ as the only shrinkers having some nontrivial translate of their spacetime track (i.e., $X_0 + \supp \cM_\Sigma$ for some $X_0 \neq (\bOh, 0)$)  contained, for sufficienty negative times, on one side of the original shrinking spacetime track (i.e., $\cup_{t < 0} \sqrt{-t} \bar \Omega$, which is a side of $\cM_\Sigma$).

\subsection{First nongeneric singular time}

We recall the following notion from \cite{CCMS:generic1}:\footnote{In the reference, $T_{\textnormal{bad}}(\cM)$ is equivalently defined as the supremum of times $T$ such that all tangent flows at $X = (\bx, t) \in \sing \cM$, $t < T$, are for negative times coincident with some $\cM_\Sigma$, $\Sigma \in \cS^{\textnormal{gen}}$.}

\begin{definition}[First nongeneric time, {cf.\ \cite[Section 11]{CCMS:generic1}}] \label{defi:first-nongeneric}
	Suppose that $\cM$ is a cyclic unit-regular integral Brakke flow in $\RR^3$ with $\cM(0) = \cH^2 \lfloor M$ for a closed embedded surface $M \subset \RR^3$. We define
	\[ T_{\textnormal{bad}}(\cM) = \inf \{ \mathfrak{t}(X) : X \in \sing \cM \setminus \sing_{\textnormal{gen}} \cM \}, \]
	with the convention $\inf \emptyset = +\infty$. 
\end{definition}

A direct corollary of Theorem \ref{theo:entropy-gap} and the upper semicontinuity of density is:

\begin{corollary} \label{coro:first-nongeneric}
	Let $\cM$ be a cyclic unit-regular integral Brakke flow with $\cM(0) = \cH^2\lfloor M$ for a closed embedded surface $M \subset \RR^3$. If $T_{\textnormal{bad}}(\cM) < \infty$, then there must exist
	\[ X \in \sing \cM \setminus \sing_{\textnormal{gen}} \cM, \; \mathfrak{t}(X) = T_{\textnormal{bad}}(\cM). \]
	Any such point $X$ is called a first nongeneric point for $\cM$.
\end{corollary} 

\subsection{Ilmanen and Bamler--Kleiner tangent flow structure theorem}

The work of Ilmanen \cite{Ilmanen:singularities} and Bamler--Kleiner \cite{BamlerKleiner} gives us a precise description of tangent flows at the first nongeneric time time:

\begin{theorem} \label{theo:ilmanen-bamler-kleiner}
	Let $\cM$ be a cyclic unit-regular integral Brakke flow with $\cM(0) = \cH^2\lfloor M$, where $M \subset \RR^3$ is a closed embedded surface. If $T_{\textnormal{bad}}(\cM) < \infty$ and $X$ is a first nongeneric point for $\cM$ (see Corollary \ref{coro:first-nongeneric}), then every tangent flow to $\cM$ at $X$ coincides, for $t < 0$, with $\cM_\Sigma$ for some $\Sigma \in \cS^* \setminus \cS^{\textnormal{gen}}$.
\end{theorem}
\begin{proof}
	The weaker result, with $\cM_\Sigma$ replaced by $k \cM_\Sigma$, $k \in \NN$, and $\Sigma \in \cS$, was proven in \cite{Ilmanen:singularities} for the first singular time. Modifications of the proof to the first nongeneric singular time were give in the proof of \cite[Proposition 11.3]{CCMS:generic1}.
	
	The fact that $k=1$ and $\Sigma \in \cS^* \setminus \cS^{\textnormal{gen}}$ follows from Bamler--Kleiner's  recent resolution of the multiplicity-one conjecture. In particular, it follows from \cite[Theorem 1.2]{BamlerKleiner} provided we can ensure $\cM$ is what Bamler--Kleiner call ``almost-regular'' on $[0, T_{\textnormal{bad}}(\cM))$. This is justified in \cite[Lemma 7.8]{BamlerKleiner}.
\end{proof}

\section{Proof of Theorem \ref{theo:main}}\label{sec:proof}

\begin{proposition} \label{prop:density-drop}
	Fix $g \in \NN$, $\Lambda > 0$. There exists $\eta_0 = \eta_0(g, \Lambda) > 0$ with the following property:
	
	Let $\Sigma \in \cS^* \setminus \cS^{\textnormal{gen}}$ and $\Omega \subset \RR^3$ be either component of $\RR^3 \setminus \Sigma$. Assume
	\begin{enumerate}
		\item[(a)] $\operatorname{genus}(\Sigma) \leq g$, 
		\item[(b)] $F(\Sigma) \leq \Lambda$.
	\end{enumerate}
	If $\cM$ is any ancient cyclic unit-regular integral Brakke flow with:
	\begin{enumerate}
		\item[(c)] $\lambda(\cM) \leq F(\Sigma)$, 
		\item[(d)] $\supp \cM \cap \mathfrak{t}^{-1}((-\infty, 0)) \subset \cup_{t < 0} \sqrt{-t} \bar \Omega$,
	\end{enumerate}
	then, $\Theta_\cM(X) \leq F(\Sigma) - \eta_0$ for every $X \in (\RR^3 \times \RR) \setminus (\bOh, 0)$.
\end{proposition}

This is our main density-drop result. It quantifies the characterization in Proposition \ref{prop:geometric-characterization}. Subject to appropriate genus and entropy bounds, ancient cyclic unit-regular integral Brakke flows on one side of a spacetime track of a non-generic shrinker have quantitatively lower densities  relative to the density of the non-generic shrinker.

\begin{proof}[Proof of Proposition \ref{prop:density-drop}]
	If not, then there exists a sequence of counterexamples $\Sigma_i, \Omega_i, \cM_i, X_i$ with
	\begin{equation} \label{eq:density-drop-1}
		\Theta_{\cM_i}(X_i) \geq F(\Sigma_i) - \tfrac{1}{i}. 
	\end{equation}
	We will use assumptions (a) and (b) to pass to limits and ensure that assumptions (c) and (d) hold on the limit. In what follows we pass to subsequences as necessary. 
	
	First, to normalize, we assume that
	\begin{equation} \label{eq:density-drop-2}
		d_p(X_i, (\bOh, 0)) = 1
	\end{equation}
	by parabolically rescaling $\cM_i$ without affecting any of the other conditions. 	By the compactness theorem for shrinkers under genus and entropy bounds from Theorem \ref{theo:shrinker-compactness}, and Allard's theorem \cite{Allard}, we deduce that 
	\begin{equation} \label{eq:density-drop-3}
		\Sigma_i \to \Sigma \in \cS^* \text{ in } C^\infty_{\textnormal{loc}},
	\end{equation}
	\begin{equation} \label{eq:density-drop-4}
		F(\Sigma_i) \to F(\Sigma).
	\end{equation}
	It follows from \eqref{eq:density-drop-4} and the Bernstein--Wang entropy gap result from Theorem \ref{theo:entropy-gap} (or by the isolatedness of $\cS^{\textnormal{gen}}$ in $C^\infty_{\textnormal{loc}}$ by Colding--Ilmanen--Minicozzi \cite{ColdingIlmanenMinicozzi}) that
	\begin{equation} \label{eq:density-drop-5}
		\Sigma \not \in \cS^{\textnormal{gen}}.
	\end{equation}
	Of course, we also have
	\begin{equation} \label{eq:density-drop-6}
		X_i \to X, \; d_p(X, (\bOh, 0)) = 1.
	\end{equation}
	Using the compactness of cyclic unit-regular integral Brakke flows (\cite{Ilmanen:elliptic} for integral Brakke flows, \cite{White:Brakke} for unit-regularity, and \cite{White:cyclic} for cyclicity) together with (c) and \eqref{eq:density-drop-4}, we can pass (c) to the limit:
	\begin{equation} \label{eq:density-drop-7}
		\cM_i \rightharpoonup \cM,
	\end{equation}
	\begin{equation} \label{eq:density-drop-8}
		\lambda(\cM) \leq F(\Sigma).
	\end{equation}
	Note that there are only two choices for each of $\Omega_i \subset \RR^3 \setminus \Sigma_i$, so we may pass to a subsequence to obtain a consistent choice among them via \eqref{eq:density-drop-3}. If $\Omega$ is the limiting component of $\RR^3 \setminus \Sigma$, we may use smooth compact mean curvature flows in $\RR^3 \setminus \bar \Omega$ as barriers in the avoidance principle to ensure that (d) passes to the limit, too:
	\begin{equation} \label{eq:density-drop-9}
		\supp \cM \cap \mathfrak{t}^{-1}((-\infty, 0)) \subset \cup_{t < 0} \sqrt{-t} \bar \Omega.
	\end{equation}

	The upper semicontinuity of densities together with \eqref{eq:density-drop-1}, \eqref{eq:density-drop-4}, \eqref{eq:density-drop-6}, \eqref{eq:density-drop-7}, \eqref{eq:density-drop-8}, implies
	\[ \Theta_\cM(X) \geq F(\Sigma) \geq \lambda(\cM). \]
	In particular, $\cM$ is a shrinking flow with spacetime center $X$ due to the monotonicity formula. If $V$ is the $F$-stationary varifold corresponding to some tangent flow to $\cM$ at $-\infty$, then
	\begin{equation} \label{eq:density-drop-10}
		F(V) = F(\Sigma),
	\end{equation}
	and by \eqref{eq:density-drop-9} and the parabolic dilation invariance of its right hand side, 
	\begin{equation} \label{eq:density-drop-11}
		\supp V \subset \bar \Omega.
	\end{equation}
	It follows from Proposition \ref{prop:frankel-shrinker}, \eqref{eq:density-drop-10}, and \eqref{eq:density-drop-11}, that $V = \Sigma$. This forces Proposition \ref{prop:geometric-characterization} into a contradiction with the fact that $\Sigma \not \in \cS^{\textnormal{gen}}$ per \eqref{eq:density-drop-5} and $X \neq (\bOh, 0)$ (since $d_p(X, (\bOh, 0)) = 1$) per \eqref{eq:density-drop-6}. This completes our proof.
\end{proof}

\begin{proof}[Proof of Theorem \ref{theo:main}]
Let $K^\circ$ be the smooth compact domain bounded by $M^\circ$, and let $(K_s)_{s \in (-1, 1)}$ be a smooth deformation of $K_0 := K^\circ$ so that
\begin{equation} \label{eq:foliation-inward}
	s_1 < s_2 \implies K_{s_1} \subset \operatorname{int} K_{s_2}.
\end{equation}
For each $s \in (-1, 1)$, denote $M_s := \partial K_s$ and let $\cM_s$ be any cyclic unit-regular integral Brakke flow with $\cM_s(0) = \cH^2 \lfloor M_s$ (see \cite{Ilmanen:elliptic}).\footnote{In our previous work, we have often had to consider multiple Brakke flows for each fixed $s \in (-1, 1)$. Our more streamlined Claim \ref{clai:density-drop} allows us to consider just a single Brakke flow for our argument.}

We may fix $\Lambda > 0$ so that
\begin{equation} \label{eq:foliation-lambda}
	\lambda(M_{s}) \leq \Lambda \text{ for all } s \in (-1, 1).
\end{equation}

\begin{claim} \label{clai:density-drop}
	Fix $s_0 \in (-1,1)$ and $X_0$ a first nongeneric point for $\cM_{s_0}$ (recall Corollary \ref{coro:first-nongeneric}). If $s_i \to s_0$, and $X_0 \neq X_i \to X_0$ are  first nongeneric points for $\cM_{s_i}$, then 
	\[ \limsup_i \Theta_{\cM_{s_i}}(X_i) \leq \Theta_{\cM_{s_0}}(X_0) - \eta_0, \]
	where $\eta_0$ is as in Proposition \ref{prop:density-drop} with $g = \operatorname{genus}(M_0)$ and $\Lambda$ as in \eqref{eq:foliation-lambda}.
\end{claim}
\begin{proof}[Proof of Claim]
	If this were false, then there would exist $\eps > 0$ so that
	\begin{equation} \label{eq:claim-1}
		\Theta_{\cM_{s_i}}(X_i) > \Theta_{\cM_{s_0}}(X_0) - \eta_0 + \eps
	\end{equation}
	for sufficiently large $i$. Translate so that $X_0 = (\bOh, 0)$ and set  $\lambda_i :=  d_p(X_i,(\bOh,0)) > 0$.	

	It follows from the Ilmanen and Bamler--Kleiner theorem on the structure of tangent flows at the first nongeneric singular time, Theorem \ref{theo:ilmanen-bamler-kleiner}, that 
	\begin{equation} \label{eq:claim-bk}
		\tilde \cM_{s_0}^i := \operatorname{ParDil}_{\lambda_i} \cM_{s_0} \rightharpoonup \tilde \cM_{s_0} = \cM_\Sigma \text{ for some } \Sigma \in \cS^* \setminus \cS^{\textnormal{gen}}.
	\end{equation}
	Note that, by the monotonicity formula,
	\begin{equation} \label{eq:claim-2}
		F(\Sigma) = \Theta_{\cM_{s_0}}(X_0) \leq \Lambda.
	\end{equation}
	By the nonfattening of the flow before $X_0$ by Hershkovits--White and Choi--Haslhofer--Hershkovits   \cite{HershkovitsWhite:nonfattening, ChoiHaslhoferHershkovits}, the regularity for a.e.~$t$ of flows with only generic singularities by \cite{White:stratification} (see also \cite{ColdingMinicozzi:sing-generic}), and White's genus monotonicity \cite{White:topology-weak},
	\begin{equation} \label{eq:claim-3}
		\operatorname{genus}(\Sigma) \leq \operatorname{genus}(M_0).
	\end{equation}
		
	Denote 
	\[ \tilde \cM^i := \operatorname{ParDil}_{\lambda_i} \cM_{s_i} \rightharpoonup \tilde \cM, \]
	\[ \operatorname{ParDil}_{\lambda_i} X_i \to \tilde X, \;  d_p(\tilde X,(\bOh,0)) = 1. \]
	The upper semicontinuity of density, \eqref{eq:claim-1}, and \eqref{eq:claim-2} imply
	\begin{equation} \label{eq:claim-4}
		\Theta_{\tilde \cM}(\tilde X) \geq F(\Sigma) - \eta_0 + \eps.
	\end{equation}
	Moreover, by a monotonicity formula argument we explain below,
	\begin{equation} \label{eq:claim-5}
		\lambda(\tilde \cM) \leq F(\Sigma).
	\end{equation}
	Indeed, if $Y \in \RR^3 \times \RR$, $r > 0$, $\rho > 0$, then we have by monotonicity and $\lambda_i \to 0$ that
	\begin{align*} 
		\Theta_{\tilde \cM}(Y, r) 
			& = \lim_i \Theta_{\tilde \cM_i}(Y, r) \\
			& = \lim_i \Theta_{\cM_i}(\operatorname{ParDil}_{\lambda_i} Y, \lambda_i r) \\
			& \leq \lim_i \Theta_{\cM_i}(\operatorname{ParDil}_{\lambda_i} Y, \rho) \\
			& \leq \Theta_{\cM}((\bOh, 0), \rho).
	\end{align*}
	Now send $\rho \to 0$ and then take $\sup$ over $Y$, $r$ to obtain \eqref{eq:claim-5}.
	
	Finally, we also have that
	\begin{equation} \label{eq:claim-6}
		\supp \tilde \cM \cap \mathfrak{t}^{-1}((-\infty, 0)) \subset \cup_{t < 0} \sqrt{-t} \bar \Omega,
	\end{equation}
	where $\Omega \subset \RR^3$ is a component of $\RR^3 \setminus \Sigma$. This follows from \eqref{eq:foliation-inward} combined with the avoidance principle and the locally smooth convergence $\tilde \cM^i_{s_0} \rightharpoonup \cM_\Sigma$ on $\mathfrak{t}^{-1}((-\infty, 0))$ from \eqref{eq:claim-bk} and Brakke's regularity theorem \cite{Brakke}, the latter of which applies since $\Sigma$ is smooth.
	
	Altogether \eqref{eq:claim-2}, \eqref{eq:claim-3}, \eqref{eq:claim-4}, \eqref{eq:claim-5}, \eqref{eq:claim-6}, and $\tilde X \neq (\bOh, 0)$ (since $d_p(\tilde X,(\bOh,0)) = 1$) contradict our choice of $\eta_0$ coming from Proposition \ref{prop:density-drop}. This completes our proof of the claim.
\end{proof}

Now proceed to define, for $\ell \in \{ 0, 1, 2, \ldots \}$, 
\begin{align*}
	\cS_\ell := \bigcup_{s \in [-1,1]} 
		& \{ X \text{ is a first nongeneric singular point for } \cM_s \\
		& \qquad \text{ and } \Theta_{\cM_s}(X) \in [1 + \tfrac12 \ell \eta_0, 1 + \tfrac12 (\ell+1)\eta_0) \}.
\end{align*}
The claim implies that $\cS_\ell$ is discrete. In particular, $\cS_\ell$ is finite since the flows become extinct after a uniform finite time. Therefore, $\cup_\ell \cS_\ell$ is also finite since the union can be taken over a finite set of $\ell$'s by \eqref{eq:foliation-lambda}. Thus, we have that:
\[ \cS = \bigcup_{s \in [-1,1]}  \{ X \text{ is a first nongeneric singular point for } \cM_s \} \]
is finite.

To conclude, choose $s_i \to 0$ so that $\cM_{s_i}$ contains no first nongeneric singular points, and thus no nongeneric singular points at all by Corollary \ref{coro:first-nongeneric}. It is known that, in the absense of nongeneric points, $\cM_{s_i}$ is the  unique cyclic unit-regular integral Brakke flow starting at $M_{s_i}$ (see Remark \ref{rema:implications} (b)). This completes the proof.
\end{proof}

\bibliographystyle{alpha}
\bibliography{main}

\end{document}